\documentclass[12pt,reqno]{amsart}
\usepackage{amsfonts, amsthm, amsmath}
\allowdisplaybreaks[4]
\usepackage{rotating}
\usepackage{tikz}
\usepackage{graphics}
\usepackage{amssymb}
\usepackage{amscd}
\usepackage{t1enc}
\usepackage[mathscr]{eucal}
\usepackage{indentfirst}
\usepackage{graphicx}
\usepackage{graphics}
\usepackage{pict2e}
\usepackage{mathrsfs}
\usepackage{enumerate}
\usepackage[pagebackref]{hyperref}
\hypersetup{backref, pagebackref, colorlinks=true}
\usepackage{cite}
\usepackage{color}
\usepackage{epic}
\usepackage{hyperref} 
\numberwithin{equation}{section}
\theoremstyle{plain}

\newtheorem{theorem}{Theorem}[section]

\newtheorem{lemma}[theorem]{Lemma}

\newtheorem{corollary}[theorem]{Corollary}

\theoremstyle{definition}

\newtheorem{?}[theorem]{Problem}

\def\boxit#1{\leavevmode\hbox{\vrule\vtop{\vbox{\kern.33333pt\hrule
    \kern1pt\hbox{\kern1pt\vbox{#1}\kern1pt}}\kern1pt\hrule}\vrule}}

\usepackage{collectbox}



\newcommand{\f}[1]{\ifthenelse{\equal{#1}{1}}{(q;q)_\infty}{(q^{#1};q^{#1})_{\infty}}}

\setbox0=\hbox{$+$}
\newdimen\plusheight
\plusheight=\ht0
\def\+{\;\lower\plusheight\hbox{$+$}\;}

\setbox0=\hbox{$-$}
\newdimen\minusheight
\minusheight=\ht0
\def\-{\;\lower\minusheight\hbox{$-$}\;}

\setbox0=\hbox{$\cdots$}
\newdimen\cdotsheight
\cdotsheight=\plusheight
\def\cds{\lower\cdotsheight\hbox{$\cdots$}}

\begin{document}

\title[Vanishing coefficients]{Vanishing coefficients in several $q$-series expansions related to the Rogers--Ramanujan continued fraction}

\author[S. Chern]{Shane Chern}
\address[Shane Chern]{Department of Mathematics, Penn State University, University Park, PA 16802, USA}
\email{shanechern@psu.edu}

\author[D. Tang]{Dazhao Tang}
\address[Dazhao Tang]{Center for Applied Mathematics, Tianjin University, Tianjin 300072, P.R. China}
\email{dazhaotang@sina.com}

\date{\today}

\begin{abstract}
	In this paper, we investigate several infinite products with vanishing Taylor coefficients in arihmetic progressions. These infinite products are closely related to the Rogers--Ramanujan continued fraction. Moreover, a handful of new identities involving Ramanujan's parameters will be established.
\end{abstract}



\maketitle

\section{Introduction}

Throughout this paper, the following customary $q$-series notation will be adopted:
\begin{align*}
(A;q)_n:\!&=\prod_{k= 0}^{n-1} (1-Aq^k),\\
(A;q)_\infty:\!&=\prod_{k= 0}^\infty (1-Aq^k),\\
(A_1,A_2,\ldots,A_n;q)_\infty:\!&=(A_1;q)_\infty (A_2;q)_\infty \cdots (A_n;q)_\infty,\\
\left(\begin{matrix}
A_1,A_2,\ldots,A_n\\
B_1,B_2,\ldots,B_m
\end{matrix};q\right)_\infty:\!&=\frac{(A_1;q)_\infty (A_2;q)_\infty \cdots (A_n;q)_\infty}{(B_1;q)_\infty (B_2;q)_\infty \cdots (B_m;q)_\infty}.
\end{align*}
For notational convenience, we also write
$$E(q):=(q;q)_\infty.$$
Let $G(q)$ and $H(q)$ be the Rogers--Ramanujan functions defined respectively by
\begin{align*}
G(q) &=\frac{1}{(q,q^4;q^5)_\infty}\\
\intertext{and}
H(q) &=\frac{1}{(q^2,q^3;q^5)_\infty}.
\end{align*}
The Rogers--Ramanujan continued fraction (with a factor of $q^{1/5}$ dropped off)
$$R(q)=\frac{1}{1}\+\frac{q}{1}\+
\frac{q^2}{1}\+\frac{q^3}{1}\+\cds$$
can be represented as the quotient of $H(q)$ and $G(q)$:
\begin{equation*}
R(q)=\frac{H(q)}{G(q)}=\left(\begin{matrix}
q,q^4\\
q^2,q^3
\end{matrix};q^5\right)_\infty.
\end{equation*}

In a recent paper of the authors \cite{CT2020}, we established 5-dissection formulas for Ramanujan's parameter
\begin{equation}\label{eq:para-k}
k(q):=qR(q)R(q^2)^2,
\end{equation}
its companion
\begin{equation}\label{eq:para-kappa}
\varkappa(q):=\frac{R(q)^2}{R(q^2)},
\end{equation}
and their reciprocals $k(q)^{-1}$ and $\varkappa(q)^{-1}$. These results can be treated as follow-ups of Hirschhorn's 5-dissections of $R(q)$ and $R(q)^{-1}$ (see \cite{Hir1998}). Notice that Ramanujan also introduced another two parameters in his lost notebook \cite[p.~13]{AB2005}:
\begin{gather}
\mu(q):=qR(q)R(q^4)
\intertext{and}
\nu(q):=\frac{R(q^{1/2})^2 R(q)}{R(q^2)}.
\end{gather}
Although the two parameters and their reciprocals do not have simple 5-dissection formulas, our numerical experiment reveals that both $\nu(q^2)$ and $\nu(q^2)^{-1}$ join the shortlist of infinite products with vanishing coefficients in arithmetic progressions. Before stating these results, let us briefly review the history of coefficient-vanishing infinite products.

In 1978, Richmond and Szekeres \cite{RS1978} first considered the infinite product
\begin{equation}\label{eq:RS}
\sum_{n=0}^\infty c(n)q^n = \left(\begin{matrix}
q^3,q^5\\
q,q^7
\end{matrix};q^8\right)_\infty
\end{equation}
and its reciprocal. One result shown by them states that $c(4n+3)$ is always zero. This paper then leaded to the work of Andrews and Bressoud \cite{AB1979}, in which it was proved that the following general infinite product
\begin{equation}\label{eq:AB}
\left(\begin{matrix}
q^r,q^{2k-r}\\
q^{k-r},q^{k+r}
\end{matrix};q^{2k}\right)_\infty
\end{equation}
shares the same coefficient-vanishing nature whenever $1\le r\le k-1$ with $k$ and $r$ coprime and of opposite parity. Other studies of coefficient-vanishing infinite products were then carried out by Alladi and Gordon \cite{AG1994}, Mc Laughlin \cite{Mac2015}, Hirschhorn \cite{Hir2018}, the second author \cite{Tan2019, Tan2020}, Baruah and Kaur \cite{BK2019}, and Dou and Xiao \cite{DX2019}.

Now we are at the position of stating the vanishing coefficients in $\nu(q^2)$ and $\nu(q^2)^{-1}$. Let us define
\begin{align}
\sum_{n=0}^\infty\alpha(n)q^n &=\nu(q^2)=\dfrac{(q,q^4;q^5)_\infty^2(-q^4,-q^6;q^{10})_\infty}
{(q^2,q^3;q^5)_\infty^2(-q^2,-q^8;q^{10})_\infty}\label{def-alpha}\\
\intertext{and}
\sum_{n=0}^\infty\beta(n)q^n &=\frac{1}{\nu(q^2)}=\dfrac{(q^2,q^3;q^5)_\infty^2(-q^2,-q^8;q^{10})_\infty}
{(q,q^4;q^5)_\infty^2(-q^4,-q^6;q^{10})_\infty}.\label{def-beta}
\end{align}

\begin{theorem}\label{THM-vanishing}
	For any $n\geq0$,
	\begin{align}
	\alpha(10n+3) &=\alpha(10n+7)=0\label{vanish-1}
	\intertext{and}
	\beta(10n+3) &=\beta(10n+7)=0.\label{vanish-2}
	\end{align}
\end{theorem}

In \cite{AB1979}, Andrews and Bressoud not only demonstrated the coefficient-vanishing property of \eqref{eq:AB}, but also obtained its $k$-dissection in terms of Lambert series. In a more recent paper of Mc Laughlin \cite{McL2019}, by the Jordan--Kronecker identity \cite[Eq.~(28.1.1)]{Hir2017}, which is a special case of Ramanujan's ${}_1\psi_1$ summation formula \cite[Eq.~(1.5)]{McL2019}, it was shown that
\begin{align}\label{Mcl-iden}
\left(\begin{matrix}
q,q,az,q/(az)\\
a,q/a,z,q/z
\end{matrix};q\right)_\infty=\sum_{j=0}^{p-1}z^j
\left(\begin{matrix}
q^p,q^p,aq^jz^p,q^{p-j}/(az^p)\\
aq^j,q^{p-j}/a,z^p,q^p/z^p
\end{matrix};q^p\right)_\infty,
\end{align}
where $p$ is a fixed positive integer; see \cite[Proposition 2.1]{McL2019}. Taking $(a,z,q)\to (q^k,q^{k-r},q^{2k})$ and $p= k$ in \eqref{Mcl-iden} shall further simplify the $k$-dissection of \eqref{eq:AB} in terms of infinite products:
\begin{align}\label{Mcl-iden-new}
&\left(\begin{matrix}
q^r,q^{2k-r}\\
q^{k-r},q^{k+r}
\end{matrix};q^{2k}\right)_\infty\left(\begin{matrix}
q^{2k},q^{2k}\\
q^k,q^k
\end{matrix};q^{2k}\right)_\infty\notag\\
&\quad=\sum_{j=0}^{k-1}q^{j(k-r)}
\left(\begin{matrix}
q^{2k^2},q^{2k^2},q^{k(k-r+2j+1)},q^{2k^2-k(k-r+2j+1)}\\
q^{k(2j+1)},q^{2k^2-k(2j+1)},q^{k(k-r)},q^{2k^2-k(k-r)}
\end{matrix};q^{2k^2}\right)_\infty.
\end{align}

Equipped with \eqref{Mcl-iden-new}, we find another two infinite products with vanishing coefficients in arithmetic progressions:
\begin{align}
\sum_{n=0}^\infty\gamma(n)q^n &=\dfrac{(-q,-q^4;q^5)_\infty^2(q^4,q^6;q^{10})_\infty}
{(-q^2,-q^3;q^5)_\infty^2(q^3,q^7;q^{10})_\infty}=\frac{R(q^2)}{R(q)^2}\left(\begin{matrix}
q^2,q^8\\
q^3,q^7
\end{matrix};q^{10}\right)_\infty\label{eq:theta}\\
\intertext{and}
\sum_{n=0}^\infty\delta(n)q^n &=\dfrac{(-q^2,-q^3;q^5)_\infty^2(q^2,q^8;q^{10})_\infty}
{(-q,-q^4;q^5)_\infty^2(q,q^9;q^{10})_\infty}=\frac{R(q)^2}{R(q^2)}\left(\begin{matrix}
q^4,q^6\\
q,q^9
\end{matrix};q^{10}\right)_\infty.\label{eq:delta}
\end{align}

\begin{theorem}\label{THM-vanishing-2}
	For any $n\geq0$,
	\begin{align}
	\gamma(5n+4) &=0\label{vanisH(q^{10})-1}
	\intertext{and}
	\delta(5n+4) &=0.\label{vanisH(q^{10})-2}
	\end{align}
\end{theorem}

The rest of this paper is organized as follows. In Sect.~\ref{sect:lemmas}, we collect
some necessary lemmas. Several identities involving Ramanujan's parameters will be discussed in Sect.~\ref{sect:thm-cor}. These identities play an important role in the proof of Theorem
\ref{THM-vanishing}, which will be presented in Sect.~\ref{sect:final-thm}. We then work on the proof of Theorem \ref{THM-vanishing-2} in Sect.~\ref{sect:final-thm-2}. Finally, we conclude in the last section with several remarks.

\section{Preliminaries}\label{sect:lemmas}

Let us collect some necessary lemmas that will be utilized in the sequel.

\begin{lemma}
	We have
	\begin{gather}
	\dfrac{E(q^2)^{10}}{E(q)^4E(q^4)^4}-\dfrac{E(q^{10})^{10}}{E(q^5)^4E(q^{20})^4}
	=\dfrac{4qE(q^2)^2E(q^5)E(q^{20})}{E(q)E(q^4)}\label{Hir-iden-4}\\
	\intertext{and}
	\dfrac{E(q^2)^4}{E(q)^2}-\dfrac{qE(q^{10})^4}{E(q^5)^2}
	=\dfrac{E(q^2)E(q^5)^3}{E(q)E(q^{10})}.\label{Hir-iden-21}
	\end{gather}
\end{lemma}

\begin{proof}
	The two identities come from (34.1.20) and (34.1.21) in \cite{Hir2017}, respectively.
\end{proof}

\begin{lemma}
	We have
	\begin{align}
	G(q)^2H(q^2)+G(q^2)H(q)^2 &=\dfrac{2G(q)G(q^2)^2E(q^{10})^2}
	{E(q^5)^2},\label{Hir-iden-1}\\
	G(q)^2H(q^2)-G(q^2)H(q)^2 &=\dfrac{2qH(q)H(q^2)^2E(q^{10})^2}
	{E(q^5)^2},\label{Hir-iden-2}\\
	G(q)G(q^2)^2-qH(q)H(q^2)^2 &=\dfrac{G(q^2)H(q)^2E(q^5)^2}
	{E(q^{10})^2}.\label{Hir-iden-3}
	\end{align}
\end{lemma}

\begin{proof}
	The identities \eqref{Hir-iden-1}--\eqref{Hir-iden-3} are
	(17.4.10), (17.4.11) and (17.4.13) in \cite{Hir2017}, respectively.
\end{proof}

\begin{lemma}
	We have
	\begin{align}
	\dfrac{R(q)^2}{R(q^2)} &=\dfrac{G(q^5)^2H(q^5)^6}{G(q^{10})H(q^{10})^3}
	-\dfrac{2qG(q^5)^4H(q^5)^3E(q^{50})^2}{G(q^{10})H(q^{10})^2E(q^{25})^2}
	+\dfrac{4q^2G(q^5)^3H(q^5)^4E(q^{50})^2}{G(q^{10})H(q^{10})^2E(q^{25})^2}\notag\\
	&\quad-\dfrac{4q^3G(q^5)^3H(q^5)^3E(q^{50})^4}{H(q^{10})^2E(q^{25})^4}
	+\dfrac{2q^4G(q^5)^3H(q^5)^4E(q^{50})^2}{G(q^{10})^2H(q^{10})E(q^{25})^2} \label{CT-iden-1}\\
	\intertext{and}
	\dfrac{R(q^2)}{R(q)^2} &=\dfrac{G(q^5)^6H(q^5)^2}{G(q^{10})^3H(q^{10})}
	+\dfrac{2qG(q^5)^4H(q^5)^3E(q^{50})^2}{G(q^{10})H(q^{10})^2E(q^{25})^2}
	-\dfrac{4q^7G(q^5)^3H(q^5)^3E(q^{50})^4}{G(q^{10})^2E(q^{25})^4}\notag\\
	&\quad-\dfrac{4q^3G(q^5)^4H(q^5)^3E(q^{50})^2}{G(q^{10})^2H(q^{10})E(q^{25})^2}
	-\dfrac{2q^4G(q^5)^3H(q^5)^4E(q^{50})^2}{G(q^{10})^2H(q^{10})E(q^{25})^2}. \label{CT-iden-2}
	\end{align}
\end{lemma}

\begin{proof}
	The identities \eqref{CT-iden-1} and \eqref{CT-iden-2}
	are equivalent to (1.7) and (1.8) in \cite{CT2020}.
\end{proof}

\begin{lemma}
	We have
	\begin{gather}
	\dfrac{1}{R(q)R(q^2)^2}-q^2R(q)R(q^2)^2
	=\dfrac{E(q^2)E(q^5)^5}{E(q)E(q^{10})^5}\label{BK-iden-1}\\
	\intertext{and}
	\dfrac{R(q^2)}{R(q)^2}-\dfrac{R(q)^2}{R(q^2)}
	=\dfrac{4qE(q)E(q^{10})^5}{E(q^2)E(q^5)^5}.\label{BK-iden-2}
	\end{gather}
\end{lemma}

\begin{proof}
	The two identities can be deduced from (1.19) and (1.20) in \cite{BB2018}.
\end{proof}

\begin{lemma}
	We have
	\begin{align}\label{key-iden-1}
	H(q)^4G(q^{2})^2-G(q)^4H(q^{2})^2+G(q)^2H(q)^2G(q^{2})H(q^{2})=
	\frac{G(q^{2})^3H(q^{2})^3}{G(q)^2H(q)^2}.
	\end{align}
\end{lemma}

\begin{proof}
	This identity appears in Theorem 3.2 of \cite{CT2020}.
\end{proof}

\section{Identities involving Ramanujan's parameters}\label{sect:thm-cor}

In this section, we establish several identities involving Ramanujan's parameters $k(q)$, $\mu(q)$ and $\nu(q)$, most of which will also be used to prove Theorem \ref{THM-vanishing}. To the best of our knowledge, these identities appear to be new.

\medskip

First, we notice that \eqref{BK-iden-1} and \eqref{BK-iden-2} can be restated in terms of $k(q)$ and $\varkappa(q)$:
\begin{align*}
\frac{1}{k(q)}-k(q)&=\dfrac{E(q^2)E(q^5)^5}{qE(q)E(q^{10})^5}\\
\intertext{and}
\frac{1}{\varkappa(q)}-\varkappa(q)&=\dfrac{4qE(q)E(q^{10})^5}{E(q^2)E(q^5)^5}.
\end{align*}
We have analogs for $\mu(q)$ and $\nu(q^2)$ also.

\begin{theorem}\label{th:iden-1}
	We have
	\begin{gather}
	\frac{1}{\mu(q)}-\mu(q)=\dfrac{E(q^2)^3E(q^{10})^5}
	{qE(q)E(q^4)E(q^5)^3E(q^{20})^3}\label{eq:mu-iden-1}\\
	\intertext{and}
	\frac{1}{\nu(q^2)}-\nu(q^2)=\dfrac{4qE(q^4)E(q^{20})^3}{E(q^{10})^4}.\label{inter-iden-1}
	\end{gather}
\end{theorem}

It is also easy to observe that $\nu(q^2) = \varkappa(q)\varkappa(q^2)$. Hence, \eqref{inter-iden-1} can be rewritten as
$$\frac{1}{\varkappa(q)\varkappa(q^2)}-\varkappa(q)\varkappa(q^2) = \dfrac{4qE(q^4)E(q^{20})^3}{E(q^{10})^4}.$$
For $k(q)$, we have an identity of similar flavor.

\begin{theorem}\label{th:iden-2}
	We have
	\begin{gather}
	\dfrac{k(q)}{k(q^2)}-
	\dfrac{k(q^2)}{k(q)}
	=\dfrac{E(q)E(q^5)^3}{qE(q^{10})^4}.\label{variant-iden-1}
	\end{gather}
\end{theorem}

Finally, we obtain two more interesting identities related to the Rogers--Ramanujan continued fraction.

\begin{theorem}\label{th:iden-3}
	We have
	\begin{align}\label{CT5-step-3}
	\dfrac{R(q)}{R(q^2)R(q^4)}-\dfrac{q^2R(q^2)R(q^4)}{R(q)} =
	\dfrac{E(q)E(q^4)E(q^{10})^{10}}{E(q^2)^2E(q^5)^5E(q^{20})^5}.
	\end{align}
	Moreover,
	\begin{align}
	\left(\dfrac{1}{R(q)R(q^4)}+q^2R(q)R(q^4)\right)-\left(\dfrac{R(q)}{R(q^2)R(q^4)}
	-\dfrac{q^2R(q^2)R(q^4)}{R(q)}\right)=2q.\label{CT-iden-5}
	\end{align}
\end{theorem}

\medskip

The proofs will be organized as follows. First, Theorem \ref{th:iden-3} will be established. We then prove Theorem \ref{th:iden-1}. Finally, the proof of Theorem \ref{th:iden-2} will be provided.

\begin{proof}[Proof of Theorem \ref{th:iden-3}]
	From \cite[Eq.~(17.4.9)]{Hir2017}, we have
	\begin{align}
	G(q)G(q^4)-qH(q)H(q^4) &=\dfrac{E(q^{10})^5}
	{E(q^2)E(q^5)^2E(q^{20})^2}.\label{Hir-iden-4.9}
	\end{align}
	Squaring both sides gives
	\begin{align*}
	G(q)^2G(q^4)^2+q^2H(q)^2H(q^4)^2 &=2qG(q)H(q)G(q^4)H(q^4)+\dfrac{E(q^{10})^{10}}{E(q^2)^2E(q^5)^4E(q^{20})^4}.
	\end{align*}
	We then divide both sides by $G(q)H(q)G(q^4)H(q^4)$. Hence,
	\begin{align}
	\dfrac{1}{R(q)R(q^4)}+q^2R(q)R(q^4)=2q+
	\dfrac{E(q)E(q^4)E(q^{10})^{10}}{E(q^2)^2E(q^5)^5E(q^{20})^5}.\label{CT5-step-2}
	\end{align}
	
	On the other hand, it follows from \eqref{BK-iden-1} and \eqref{BK-iden-2} that
	\begin{align*}
	&\dfrac{4qE(q)E(q^4)E(q^{10})^{10}}{E(q^2)^2E(q^5)^5E(q^{20})^5}
	=\dfrac{4qE(q)E(q^{10})^5}{E(q^2)E(q^5)^5}\cdot
	\dfrac{E(q^4)E(q^{10})^5}{E(q^2)E(q^{20})^5}\\
	&\quad=\left(\dfrac{R(q^2)}{R(q)^2}-\dfrac{R(q)^2}{R(q^2)}\right)\cdot
	\left(\dfrac{1}{R(q^2)R(q^4)^2}-q^4R(q^2)R(q^4)^2\right)\\
	&\quad=\left(\dfrac{1}{R(q)^2R(q^4)^2}+q^4R(q)^2R(q^4)^2\right)
	-\left(\dfrac{R(q)^2}{R(q^2)^2R(q^4)^2}
	+\dfrac{q^4R(q^2)^2R(q^4)^2}{R(q)^2}\right)\\
	&\quad=\left(\dfrac{1}{R(q)R(q^4)}+q^2R(q)R(q^4)\right)^2
	-\left(\dfrac{R(q)}{R(q^2)R(q^4)}-\dfrac{q^2R(q^2)R(q^4)}{R(q)}\right)^2
	-4q^2.
	\end{align*}
	Substituting \eqref{CT5-step-2} into this identity gives
	\begin{align*}
	\dfrac{4qE(q)E(q^4)E(q^{10})^{10}}{E(q^2)^2E(q^5)^5E(q^{20})^5}
	&=\left(\dfrac{E(q)E(q^4)E(q^{10})^{10}}{E(q^2)^2E(q^5)^5E(q^{20})^5}\right)^2
	+\dfrac{4qE(q)E(q^4)E(q^{10})^{10}}{E(q^2)^2E(q^5)^5E(q^{20})^5}\\
	&\quad-\left(\dfrac{R(q)}{R(q^2)R(q^4)}-\dfrac{q^2R(q^2)R(q^4)}{R(q)}\right)^2,
	\end{align*}
	namely,
	\begin{align*}
	\left(\dfrac{R(q)}{R(q^2)R(q^4)}-\dfrac{q^2R(q^2)R(q^4)}{R(q)}\right)^2
	=\left(\dfrac{E(q)E(q^4)E(q^{10})^{10}}{E(q^2)^2E(q^5)^5E(q^{20})^5}\right)^2.
	\end{align*}
	We therefore obtain \eqref{CT5-step-3} by equating the constant term. Further, \eqref{CT-iden-5} follows from \eqref{CT5-step-3} and \eqref{CT5-step-2}.
\end{proof}

\begin{proof}[Proof of Theorem \ref{th:iden-1}]
	We first restate \eqref{eq:mu-iden-1} and \eqref{inter-iden-1} as
	\begin{gather}
	\dfrac{1}{R(q)R(q^4)}-q^2R(q)R(q^4)=\dfrac{E(q^2)^3E(q^{10})^5}
	{E(q)E(q^4)E(q^5)^3E(q^{20})^3}\label{eq:mu-iden}\\
	\intertext{and}
	\dfrac{R(q^4)}{R(q)^2 R(q^2)}
	-\dfrac{R(q)^2 R(q^2)}{R(q^4)}=\dfrac{4qE(q^4)E(q^{20})^3}{E(q^{10})^4}.\label{inter-iden}
	\end{gather}
	
	Let us square both sides of \eqref{CT5-step-2} and subtract $4q^2$. Then
	\begin{align*}
	\left(\dfrac{1}{R(q)R(q^4)}-q^2R(q)R(q^4)\right)^2&=\dfrac{4qE(q)E(q^4)E(q^{10})^{10}}{E(q^2)^2E(q^5)^5E(q^{20})^5}\\
	&\quad+
	\left(\dfrac{E(q)E(q^4)E(q^{10})^{10}}{E(q^2)^2E(q^5)^5E(q^{20})^5}\right)^2.
	\end{align*}
	It follows from \eqref{Hir-iden-4} that
	\begin{align*}
	&\dfrac{4qE(q)E(q^4)E(q^{10})^{10}}{E(q^2)^2E(q^5)^5E(q^{20})^5}=\frac{E(q)^2 E(q^4)^2 E(q^{10})^{10}}{E(q^2)^4 E(q^5)^6 E(q^{20})^6}\cdot \dfrac{4qE(q^2)^2E(q^5)E(q^{20})}{E(q)E(q^4)}\\
	&\quad=\frac{E(q)^2 E(q^4)^2 E(q^{10})^{10}}{E(q^2)^4 E(q^5)^6 E(q^{20})^6}\left(\dfrac{E(q^2)^{10}}{E(q)^4E(q^4)^4}-\dfrac{E(q^{10})^{10}}{E(q^5)^4E(q^{20})^4}\right)\\
	&\quad=\left(\dfrac{E(q^2)^3E(q^{10})^5}
	{E(q)E(q^4)E(q^5)^3E(q^{20})^3}\right)^2
	-\left(\dfrac{E(q)E(q^4)E(q^{10})^{10}}{E(q^2)^2E(q^5)^5E(q^{20})^5}\right)^2.
	\end{align*}
	Hence,
	\begin{align*}
	\left(\dfrac{1}{R(q)R(q^4)}-q^2R(q)R(q^4)\right)^2=\left(\dfrac{E(q^2)^3E(q^{10})^5}
	{E(q)E(q^4)E(q^5)^3E(q^{20})^3}\right)^2.
	\end{align*}
	We therefore obtain \eqref{eq:mu-iden}.
	
	\medskip
	
	Now we turn to prove \eqref{inter-iden}.

    Multiplying \eqref{BK-iden-1} by \eqref{CT5-step-3} gives
	\begin{align*}
	&\dfrac{E(q^4)E(q^{10})^5}{E(q^2)E(q^{20})^5}
	=\dfrac{E(q)E(q^4)E(q^{10})^{10}}{E(q^2)^2E(q^5)^5E(q^{20})^5}
	\cdot\dfrac{E(q^2)E(q^5)^5}{E(q)E(q^{10})^5}\notag\\
	&\quad=\left(\dfrac{R(q)}{R(q^2)R(q^4)}-\dfrac{q^2R(q^2)R(q^4)}{R(q)}\right)\cdot
	\left(\dfrac{1}{R(q)R(q^2)^2}-q^2R(q)R(q^2)^2\right)\notag\\
	&\quad=\left(\dfrac{1}{R(q^2)^3R(q^4)}+q^4R(q^2)^3R(q^4)\right)
	-q^2\left(\dfrac{R(q^4)}{R(q)^2R(q^2)}+\dfrac{R(q)^2R(q^2)}{R(q^4)}\right).
	\end{align*}
	From \cite[Eq.~(1.22)]{BB2018} (see also \cite[Theorem 1.1]{CT2019}),
	\begin{align*}
	\dfrac{1}{R(q^2)^3R(q^4)}+q^4R(q^2)^3R(q^4)=
	\dfrac{4q^4E(q^2)E(q^{20})^5}{E(q^4)E(q^{10})^5}
	+\dfrac{E(q^4)E(q^{10})^5}{E(q^2)E(q^{20})^5}+2q^2.
	\end{align*}
	Hence,
	\begin{align*}
	\dfrac{R(q^4)}{R(q)^2R(q^2)}+\dfrac{R(q)^2R(q^2)}{R(q^4)}
	=\dfrac{4q^2E(q^2)E(q^{20})^5}{E(q^4)E(q^{10})^5}+2,
	\end{align*}
	from which we obtain
	\begin{align*}
	\left(\dfrac{R(q^4)}{R(q)^2R(q^2)}-\dfrac{R(q)^2R(q^2)}{R(q^4)}\right)^2 &=
	\left(\dfrac{R(q^4)}{R(q)^2R(q^2)}+\dfrac{R(q)^2R(q^2)}{R(q^4)}\right)^2-4\\
	&=\dfrac{16q^2E(q^2)E(q^{20})^5}{E(q^4)E(q^{10})^5}
	+\dfrac{16q^4E(q^2)^2E(q^{20})^{10}}{E(q^4)^2E(q^{10})^{10}}.
	\end{align*}
	We further rewrite \eqref{Hir-iden-21} as
	$$1+\dfrac{q^2E(q^2)E(q^{20})^5}{E(q^4)E(q^{10})^5}=
	\dfrac{E(q^4)^3E(q^{20})}{E(q^2)E(q^{10})^3}.$$
	Hence,
	\begin{align*}
	\left(\dfrac{R(q^4)}{R(q)^2R(q^2)}-\dfrac{R(q)^2R(q^2)}{R(q^4)}\right)^2 &=
	\dfrac{16q^2E(q^2)E(q^{20})^5}{E(q^4)E(q^{10})^5}
	\left(1+\dfrac{q^2E(q^2)E(q^{20})^5}{E(q^4)E(q^{10})^5}\right)\\
	&=\dfrac{16q^2E(q^2)E(q^{20})^5}{E(q^4)E(q^{10})^5}\cdot
	\dfrac{E(q^4)^3E(q^{20})}{E(q^2)E(q^{10})^3}\\
	&=\left(\dfrac{4qE(q^4)E(q^{20})^3}{E(q^{10})^4}\right)^2,
	\end{align*}
	from which we obtain \eqref{inter-iden}.
\end{proof}

\begin{proof}[Proof of Theorem \ref{th:iden-2}]
	The proof of \eqref{variant-iden-1} is similar to that of \eqref{inter-iden-1}. We first restate it as
	\begin{equation}\label{variant-iden}
	\dfrac{R(q)R(q^2)}{R(q^4)^2}-
	\dfrac{q^2R(q^4)^2}{R(q)R(q^2)}
	=\dfrac{E(q)E(q^5)^3}{E(q^{10})^4}.
	\end{equation}
	According to \eqref{BK-iden-2} and \eqref{CT5-step-3}, one has
	\begin{align*}
	&\dfrac{4q^2E(q)E(q^{10})^5}{E(q^2)E(q^5)^5}
	=\dfrac{4q^2E(q^2)E(q^{20})^5}{E(q^4)E(q^{10})^5}\cdot
	\dfrac{E(q)E(q^4)E(q^{10})^{10}}{E(q^2)^2E(q^5)^5E(q^{20})^5}\notag\\
	&\quad=\left(\dfrac{R(q)}{R(q^2)R(q^4)}-\dfrac{q^2R(q^2)R(q^4)}{R(q)}\right)\cdot
	\left(\dfrac{R(q^4)}{R(q^2)^2}-\dfrac{R(q^2)^2}{R(q^4)}\right)\notag\\
	&\quad=\left(\dfrac{R(q)}{R(q^2)^3}+\dfrac{q^2R(q^2)^3}{R(q)}\right)
	-\left(\dfrac{R(q)R(q^2)}{R(q^4)^2}+
	\dfrac{q^2R(q^4)^2}{R(q)R(q^2)}\right).
	\end{align*}
	From \cite[Eq.~(1.21)]{BB2018} (see also \cite[Theorem 1.1]{CT2019}),
	\begin{align}
	\dfrac{R(q)}{R(q^2)^3}+\dfrac{q^2R(q^2)^3}{R(q)}
	=\dfrac{E(q^2)E(q^5)^5}{E(q)E(q^{10})^5}
	+\dfrac{4q^2E(q)E(q^{10})^5}{E(q^2)E(q^5)^5}-2q.\label{2q-iden-2}
	\end{align}
	Hence,
	\begin{align*}
	\dfrac{R(q)R(q^2)}{R(q^4)^2}+\dfrac{q^2R(q^4)^2}{R(q)R(q^2)}
	=\dfrac{E(q^2)E(q^5)^5}{E(q)E(q^{10})^5}-2q,
	\end{align*}
	from which we obtain
	\begin{align*}
	\left(\dfrac{R(q)R(q^2)}{R(q^4)^2}-\dfrac{q^2R(q^4)^2}{R(q)R(q^2)}\right)^2
	&=\left(\dfrac{R(q)R(q^2)}{R(q^4)^2}+\dfrac{q^2R(q^4)^2}{R(q)R(q^2)}\right)^2-4q^2\\
	&=\dfrac{E(q^2)^2E(q^5)^{10}}{E(q)^2E(q^{10})^{10}}
	-4q\dfrac{E(q^2)E(q^5)^5}{E(q)E(q^{10})^5}.
	\end{align*}
	Replacing $q$ by $-q$ in \eqref{Hir-iden-4}, we find that
	$$\dfrac{E(q^2)E(q^5)^5}{E(q)E(q^{10})^5}-4q=\dfrac{E(q)^3E(q^5)}{E(q^2)E(q^{10})^3}.$$
	Hence,
	\begin{align*}
	\left(\dfrac{R(q)R(q^2)}{R(q^4)^2}-\dfrac{q^2R(q^4)^2}{R(q)R(q^2)}\right)^2
	&=\dfrac{E(q^2)E(q^5)^5}{E(q)E(q^{10})^5}
	\left(\dfrac{E(q^2)E(q^5)^5}{E(q)E(q^{10})^5}-4q\right)\\
	&=\dfrac{E(q^2)E(q^5)^5}{E(q)E(q^{10})^5}\cdot
	\dfrac{E(q)^3E(q^5)}{E(q^2)E(q^{10})^3}\\
	&=\left(\dfrac{E(q)E(q^5)^3}{E(q^{10})^4}\right)^2,
	\end{align*}
	from which \eqref{variant-iden} follows.
\end{proof}

\section{Proof of Theorem \ref{THM-vanishing}}\label{sect:final-thm}

We first require a corollary of \eqref{inter-iden-1}.

\begin{corollary}\label{cor-inter-iden}
	For any $n\geq0$,
	\begin{align}
	\alpha(5n+2) &=\beta(5n+2)\label{int-iden-1}\\
	\intertext{and}
	\alpha(5n+3) &=\beta(5n+3).\label{int-iden-2}
	\end{align}
\end{corollary}

\begin{proof}
	Notice that \eqref{inter-iden-1} can be restated as
	$$\sum_{n=0}^\infty\beta(n)q^n-\sum_{n=0}^\infty\alpha(n)q^n=\dfrac{4qE(q^4)E(q^{20})^3}{E(q^{10})^4}.$$
	Euler's Pentagonal Numbers Theorem \cite[Eq.~(1.6.1)]{Hir2017} tells us that there are no terms of the form $q^{5n+3}$ and $q^{5n+4}$ in the expansion of $E(q)=(q;q)_\infty$. Hence, there are no terms of the form $q^{5n+2}$ and $q^{5n+3}$ in the expansion of $qE(q^4)$. Our desired results therefore hold.
\end{proof}

The relations \eqref{int-iden-1} and \eqref{int-iden-2} clearly imply that if we can prove one of \eqref{vanish-1} and \eqref{vanish-2}, the other holds automatically. Here we will show \eqref{vanish-1}. Our proof relies on the following result.

\begin{theorem}
	We have
	\begin{align}
	\sum_{n=0}^\infty\alpha(5n+2)q^n&=\dfrac{2G(q^2)^4H(q^2)^3E(q^{20})^2}
	{G(q^4)H(q^4)^2E(q^{10})^2}\label{A2}\\
	\intertext{and}
	\sum_{n=0}^\infty\alpha(5n+3)q^n&=-\dfrac{2qG(q^2)^3H(q^2)^4E(q^{20})^2}{G(q^4)^2H(q^4)E(q^{10})^2}.\label{A3}
	\end{align}
\end{theorem}

For its proof, we begin with an auxiliary identity.

\begin{lemma}\label{usful-lemma-2}
	We have
	\begin{align}
	\dfrac{2qG(q)^3H(q)^4H(q^2)^4G(q^4)E(q^{10})^4}
	{E(q^5)^2E(q^{20})^2}&=2qG(q)^4H(q)^4G(q^2)H(q^2)G(q^4)H(q^4)\notag\\
	&\quad-2q^2G(q)^5H(q)^3H(q^2)^2H(q^4)^2.\label{0-iden}
	\end{align}
\end{lemma}

\begin{proof}
	We know from \eqref{CT-iden-5} that
	\begin{align*}
	0&=G(q)^4H(q)^4G(q^2)H(q^2)G(q^4)H(q^4)\\
	&\quad\times\Bigg(\left(\dfrac{1}{R(q)R(q^4)}+q^2R(q)R(q^4)\right)-\left(\dfrac{R(q)}{R(q^2)R(q^4)}
	-\dfrac{q^2R(q^2)R(q^4)}{R(q)}\right)-2q\Bigg)\\
	&=G(q)^3H(q)^3\left(G(q)^2H(q^2)-G(q^2)H(q)^2\right)\left(G(q^2)G(q^4)^2-q^2H(q^2)H(q^4)^2\right)\\
	&\quad -2qG(q)^4H(q)^4G(q^2)H(q^2)G(q^4)H(q^4)+2q^2G(q)^5H(q)^3H(q^2)^2H(q^4)^2.
	\end{align*}
	In view of \eqref{Hir-iden-2} and \eqref{Hir-iden-3}, we further find that
	\begin{align*}
	\dfrac{2qG(q)^3H(q)^4H(q^2)^4G(q^4)E(q^{10})^4}
	{E(q^5)^2E(q^{20})^2}
	&=G(q)^3H(q)^3\left(G(q)^2H(q^2)-G(q^2)H(q)^2\right)\\
	&\quad\times\left(G(q^2)G(q^4)^2-q^2H(q^2)H(q^4)^2\right).
	\end{align*}
	Our desired result therefore follows.
\end{proof}

It follows from \eqref{CT-iden-1} that
\begin{align}
\sum_{n=0}^\infty\alpha(n)q^n
&=\dfrac{R(q)^2}{R(q^2)}\cdot\dfrac{R(q^2)^2}{R(q)^4}\notag\\
&=\Bigg(\dfrac{G(q^5)^2H(q^5)^6}{G(q^{10})H(q^{10})^3}
-\dfrac{2qG(q^5)^4H(q^5)^3E(q^{50})^2}{G(q^{10})H(q^{10})^2E(q^{25})^2}
+\dfrac{4q^2G(q^5)^3H(q^5)^4E(q^{50})^2}{G(q^{10})H(q^{10})^2E(q^{25})^2}\notag\\
&\quad-\dfrac{4q^3G(q^5)^3H(q^5)^3E(q^{50})^4}{H(q^{10})^2E(q^{25})^4}
+\dfrac{2q^4G(q^5)^3H(q^5)^4E(q^{50})^2}{G(q^{10})^2H(q^{10})E(q^{25})^2}
\Bigg)\notag\\
&\times\Bigg(\dfrac{G(q^{10})^2H(q^{10})^6}{G(q^{20})H(q^{20})^3}
-\dfrac{2q^2G(q^{10})^4H(q^{10})^3E(q^{100})^2}
{G(q^{20})H(q^{20})^2E(q^{50})^2}\notag\\
&\quad+\dfrac{4q^4G(q^{10})^3H(q^{10})^4E(q^{100})^2}
{G(q^{20})H(q^{20})^2E(q^{50})^2}
-\dfrac{4q^6G(q^{10})^3H(q^{10})^3E(q^{100})^4}
{H(q^{20})^2E(q^{50})^4}\notag\\
&\quad+\dfrac{2q^8G(q^{10})^3H(q^{10})^4E(q^{100})^2}
{G(q^{20})^2H(q^{20})E(q^{50})^2}\Bigg).\label{iden-5A}
\end{align}

\begin{proof}[Proof of \eqref{A2}]
	We deduce from \eqref{iden-5A} that
	\begin{align*}
	A_2(q):\!&=\sum_{n=0}^\infty\alpha(5n+2)q^n\\
	&=\dfrac{4G(q)^3H(q)^4H(q^2)^4G(q^4)E(q^{10})^2}{E(q^5)^2}\\
	&\quad-\dfrac{2G(q)^2H(q)^6G(q^2)^2G(q^4)H(q^4)E(q^{20})^2}
	{E(q^{10})^2}\\
	&\quad-\dfrac{16qG(q)^3H(q)^3G(q^2)^2H(q^2)^2
		G(q^4)H(q^4)E(q^{10})^2E(q^{20})^2}{E(q^5)^4}\\
	&\quad+\dfrac{8qG(q)^4H(q)^3G(q^2)H(q^2)G(q^4)^2H(q^4)E(q^{20})^4}
	{E(q^5)^2E(q^{10})^2}\\
	&\quad+\dfrac{4q^2G(q)^3H(q)^4H(q^2)^3H(q^4)^2E(q^{20})^2}{E(q^5)^2}.
	\end{align*}
	With the aid of \eqref{Hir-iden-1} and \eqref{Hir-iden-2}, one has
	\begin{align*}
	&\dfrac{G(q^4)^2H(q^4)^3E(q^{10})^2}{G(q^2)E(q^{20})^2}\cdot A_2(q)\\
	&\quad=\dfrac{4G(q)^3H(q)^4H(q^2)^4
		G(q^4)E(q^{10})^4}{E(q^5)^2E(q^{20})^2}\\
	&\quad\quad-2G(q)^2H(q)^6G(q^2)^2G(q^4)H(q^4)\\
	&\quad\quad-4G(q)^2H(q)^2G(q^4)H(q^4)
	\left(G(q)^4H(q^2)^2-G(q^2)^2H(q)^4\right)\\
	&\quad\quad+\dfrac{4qG(q)^4H(q)^3H(q^2)H(q^4)E(q^{10})^2}
	{E(q^5)^2}\left(G(q^2)^2H(q^4)+G(q^4)H(q^2)^2\right)\\
	&\quad\quad+2qG(q)^3H(q)^3H(q^2)H(q^4)^2
	\left(G(q)^2H(q^2)-G(q^2)H(q)^2\right).
	\end{align*}
	We then tactfully rewrite the second last term on the right-hand side using \eqref{Hir-iden-1} and \eqref{Hir-iden-2}:
	\begin{align*}
	&\dfrac{4qG(q)^4H(q)^3H(q^2)H(q^4)E(q^{10})^2}
	{E(q^5)^2}\left(G(q^2)^2H(q^4)+G(q^4)H(q^2)^2\right)\\
	&\quad = 2qG(q)^3H(q)^3H(q^2)H(q^4)^2
	\left(G(q)^2H(q^2)+G(q^2)H(q)^2\right)\\
	&\quad\quad+2G(q)^4H(q)^2H(q^2)G(q^4)H(q^4)
	\left(G(q)^2H(q^2)-G(q^2)H(q)^2\right).
	\end{align*}
	Hence,
	\begin{align*}
	&\dfrac{G(q^4)^2H(q^4)^3E(q^{10})^2}{G(q^2)E(q^{20})^2}\cdot A_2(q)\\
	&\quad=\dfrac{4G(q)^3H(q)^4H(q^2)^4G(q^4)E(q^{10})^4}
	{E(q^5)^2E(q^{20})^2}\\
	&\quad\quad-2G(q)^2H(q)^6G(q^2)^2G(q^4)H(q^4)\\
	&\quad\quad-4G(q)^2H(q)^2G(q^4)H(q^4)
	\left(G(q)^4H(q^2)^2-G(q^2)^2H(q)^4\right)\\
	&\quad\quad+2qG(q)^3H(q)^3H(q^2)H(q^4)^2
	\left(G(q)^2H(q^2)+G(q^2)H(q)^2\right)\\
	&\quad\quad+2G(q)^4H(q)^2H(q^2)G(q^4)H(q^4)
	\left(G(q)^2H(q^2)-G(q^2)H(q)^2\right)\\
	&\quad\quad+2qG(q)^3H(q)^3H(q^2)H(q^4)^2
	\left(G(q)^2H(q^2)-G(q^2)H(q)^2\right).
	\end{align*}
	By virtue of \eqref{key-iden-1}, after simplification, one has
	\begin{align*}
	A_2(q) &=\dfrac{2G(q^2)E(q^{20})^2}{qG(q^4)^2H(q^4)^3E(q^{10})^2}\\
	&\times\Bigg(\dfrac{2qG(q)^3H(q)^4H(q^2)^4
		G(q^4)E(q^{10})^4}{E(q^5)^2E(q^{20})^2}\\
	&\quad-2qG(q)^4H(q)^4G(q^2)H(q^2)G(q^4)H(q^4)+2q^2G(q)^5H(q)^3H(q^2)^2H(q^4)^2\Bigg)\\
	&+\dfrac{2G(q^2)^4H(q^2)^3E(q^{20})^2}
	{G(q^4)H(q^4)^2E(q^{10})^2}.
	\end{align*}
	Finally, \eqref{A2} follows by making use of \eqref{0-iden}.
\end{proof}

\begin{proof}[Proof of \eqref{A3}]
	We deduce from \eqref{iden-5A} that
	\begin{align*}
	A_3(q):=\sum_{n=0}^\infty\alpha(5n+3)q^n&=-\dfrac{4G(q)^3H(q)^3G(q^2)^2H(q^2)^4E(q^{10})^4}
	{G(q^4)H(q^4)^3E(q^5)^4}\\
	&\quad+\dfrac{4G(q)^4H(q)^3G(q^2)^3H(q^2)E(q^{20})^2}
	{G(q^4)H(q^4)^2E(q^5)^2}\\
	&\quad+\dfrac{8qG(q)^3H(q)^4G(q^2)H(q^2)^3E(q^{20})^2}
	{G(q^4)H(q^4)^2E(q^5)^2}\\
	&\quad-\dfrac{16qG(q)^3H(q)^4G(q^2)^2H(q^2)E(q^{10})^6}
	{H(q^4)^2E(q^5)^2E(q^{20})^4}\\
	&\quad+\dfrac{2qG(q)^2H(q)^6G(q^2)^2H(q^2)E(q^{20})^2}
	{G(q^4)^2H(q^4)E(q^{10})^2}.
	\end{align*}
	Applying \eqref{Hir-iden-1} and \eqref{Hir-iden-2}, we have
	\begin{align*}
	A_3(q) &=-\dfrac{4G(q)^3H(q)^3G(q^2)^2H(q^2)^4E(q^{10})^4}
	{G(q^4)H(q^4)^3E(q^5)^4}\\
	&\quad+\dfrac{4G(q)^4H(q)^3G(q^2)^3H(q^2)E(q^{20})^2}
	{G(q^4)H(q^4)^2E(q^5)^2}\\
	&\quad-\dfrac{H(q^2)E(q^{20})^2}{G(q^4)^2H(q^4)^2E(q^{10})^2}\\
	&\qquad\times\Bigg(4G(q)^3H(q)^3G(q^2)G(q^4)
	\left(-G(q)^2H(q^2)+G(q^2)H(q)^2\right)\\
	&\qquad\quad+\dfrac{8qG(q)^3H(q)^4G(q^2)E(q^{10})^2
		\left(G(q^2)^2H(q^4)+G(q^4)H(q^2)^2\right)}{E(q^5)^2}\\
	&\qquad\quad-2qG(q)^2H(q)^6G(q^2)^2H(q^4)\Bigg)\\
	&=-\dfrac{4G(q)^3H(q)^3G(q^2)^2H(q^2)^4E(q^{10})^4}
	{G(q^4)H(q^4)^3E(q^5)^4}\notag\\
	&\quad+\dfrac{4G(q)^4H(q)^3G(q^2)^3H(q^2)E(q^{20})^2}
	{G(q^4)H(q^4)^2E(q^5)^2}\notag\\
	&\quad-\dfrac{H(q^2)E(q^{20})^2}
	{G(q^4)^2H(q^4)^2E(q^{10})^2}\notag\\
	&\qquad\times\Big(4G(q)^3H(q)^3G(q^2)G(q^4)
	\left(-G(q)^2H(q^2)+G(q^2)H(q)^2\right)\notag\\
	&\qquad\quad+4qG(q)^2H(q)^4G(q^2)H(q^4)
	\left(G(q)^2H(q^2)+G(q^2)H(q)^2\right)\notag\\
	&\qquad\quad+4G(q)^3H(q)^3G(q^2)G(q^4)
	\left(G(q)^2H(q^2)-G(q^2)H(q)^2\right)\notag\\
	&\qquad\quad-2qG(q)^2H(q)^6G(q^2)^2H(q^4)\Big).
	\end{align*}
	Substituting \eqref{key-iden-1} into the above identity and using \eqref{Hir-iden-1}, we obtain
	\begin{align*}
	A_3(q) &=-\dfrac{2qG(q^2)^3H(q^2)^4E(q^{20})^2}
	{G(q^4)^2H(q^4)E(q^{10})^2}
	-\dfrac{4G(q)^3H(q)^3G(q^2)^2H(q^2)^4E(q^{10})^4}
	{G(q^4)H(q^4)^3E(q^5)^4}\notag\\
	&\quad+\dfrac{4G(q)^4H(q)^3G(q^2)^3H(q^2)E(q^{20})^2}
	{G(q^4)H(q^4)^2E(q^5)^2}\notag\\
	&\quad
	-\dfrac{2qG(q)^4H(q)^2H(q^2)^2E(q^{20})^2}
	{G(q^4)^2H(q^4)E(q^{10})^2}\left(G(q)^2H(q^2)+G(q^2)H(q)^2\right)\notag\\
	&=-\dfrac{2qG(q^2)^3H(q^2)^4E(q^{20})^2}{G(q^4)^2H(q^4)E(q^{10})^2}
	-\dfrac{2G(q^2)^2E(q^{20})^2}{qH(q)G(q^4)^2H(q^4)^3E(q^5)^2}\notag\\
	&\qquad\times\Bigg(\dfrac{2qG(q)^3H(q)^4
		H(q^2)^4G(q^4)E(q^{10})^4}{E(q^5)^2E(q^{20})^2}\notag\\
	&\qquad\quad-2qG(q)^4H(q)^4G(q^2)H(q^2)G(q^4)H(q^4)\notag\\
	&\qquad\quad+2q^2G(q)^5H(q)^3H(q^2)^2H(q^4)^2\Bigg).
	\end{align*}
	Finally, utilizing \eqref{0-iden} yields \eqref{A3}.
\end{proof}

At the end of this section, we complete the proof of \eqref{vanish-1}.

\begin{proof}[Proof of \eqref{vanish-1}]
	It is a trivial observation that there are no terms of the form $q^{2n+1}$ in the expansion of the right-hand side of \eqref{A2}. Hence, $\alpha(10n+7)=0$. Similarly, there are no terms of the form $q^{2n}$ in the expansion of the right-hand side of \eqref{A3}. This implies that $\alpha(10n+3)=0$.
\end{proof}

\section{Proof of Theorem \ref{THM-vanishing-2}}\label{sect:final-thm-2}

We need the following two necessary identities.

\begin{lemma}
	We have
	\begin{align}
	\left(\begin{matrix}
	q^2,q^8\\
	q^3,q^7
	\end{matrix};q^{10}\right)_\infty & = \frac{G(q^{10})^3 H(q^{10})^2 E(q^{25})^2}{G(q^{5}) H(q^{5}) E(q^{50})^2} + \frac{q^6 G(q^{10})^2 H(q^{10})^4}{G(q^{5})^2 H(q^{5})}\notag\\
	&\quad - \frac{
		q^2 G(q^{10})^2 H(q^{10})^3 E(q^{25})^2}{G(q^{5}) H(q^{5}) E(q^{50})^2} + \frac{q^3 G(q^{10})^2 H(q^{10})^3 E(q^{25})^2}{G(q^{5})^2 E(q^{50})^2}\label{eq:AB-5.2}\\
	\intertext{and}
	\left(\begin{matrix}
	q^4,q^6\\
	q,q^9
	\end{matrix};q^{10}\right)_\infty & = \frac{G(q^{10})^3 H(q^{10})^2 E(q^{25})^2}{H(q^{5})^2 E(q^{50})^2} + \frac{q G(q^{10})^3 H(q^{10})^2 E(q^{25})^2}{G(q^{5}) H(q^{5}) E(q^{50})^2}\notag\\
	&\quad + \frac{
		q^2 G(q^{10})^4 H(q^{10})^2}{G(q^{5}) H(q^{5})^2} + \frac{q^3 G(q^{10})^2 H(q^{10})^3 E(q^{25})^2}{G(q^{5}) H(q^{5}) E(q^{50})^2}.\label{eq:AB-5.4}
	\end{align}
\end{lemma}

\begin{proof}
	For \eqref{eq:AB-5.2}, we set $k=5$ and $r=2$ in \eqref{Mcl-iden-new}. For \eqref{eq:AB-5.4}, we take $k=5$ and $r=4$.
\end{proof}

Now we are in the position of proving Theorem \ref{THM-vanishing-2}. First, multiplying \eqref{CT-iden-2} by \eqref{eq:AB-5.2} gives
\begin{align*}
\sum_{n=0}^\infty\gamma(n)q^n & = \Bigg(\dfrac{G(q^5)^6H(q^5)^2}{G(q^{10})^3H(q^{10})}
+\dfrac{2qG(q^5)^4H(q^5)^3E(q^{50})^2}{G(q^{10})H(q^{10})^2E(q^{25})^2}\\
&\quad\quad
-\dfrac{4q^7G(q^5)^3H(q^5)^3E(q^{50})^4}{G(q^{10})^2E(q^{25})^4}
-\dfrac{4q^3G(q^5)^4H(q^5)^3E(q^{50})^2}{G(q^{10})^2H(q^{10})E(q^{25})^2}\\
&\quad\quad
-\dfrac{2q^4G(q^5)^3H(q^5)^4E(q^{50})^2}{G(q^{10})^2H(q^{10})E(q^{25})^2}\Bigg)\\
&\times\Bigg(\frac{G(q^{10})^3 H(q^{10})^2 E(q^{25})^2}{G(q^{5}) H(q^{5}) E(q^{50})^2} + \frac{q^6 G(q^{10})^2 H(q^{10})^4}{G(q^{5})^2 H(q^{5})}\\
&\quad\quad - \frac{
	q^2 G(q^{10})^2 H(q^{10})^3 E(q^{25})^2}{G(q^{5}) H(q^{5}) E(q^{50})^2} + \frac{q^3 G(q^{10})^2 H(q^{10})^3 E(q^{25})^2}{G(q^{5})^2 E(q^{50})^2}\Bigg).
\end{align*}
Hence,
\begin{align*}
\sum_{n=0}^\infty\gamma(5n+4)q^{5n+4} & = \dfrac{2qG(q^5)^4H(q^5)^3E(q^{50})^2}{G(q^{10})H(q^{10})^2E(q^{25})^2}\cdot \frac{q^3 G(q^{10})^2 H(q^{10})^3 E(q^{25})^2}{G(q^{5})^2 E(q^{50})^2}\\
&\quad+ \dfrac{4q^7G(q^5)^3H(q^5)^3E(q^{50})^4}{G(q^{10})^2E(q^{25})^4}\cdot \frac{
	q^2 G(q^{10})^2 H(q^{10})^3 E(q^{25})^2}{G(q^{5}) H(q^{5}) E(q^{50})^2}\\
&\quad- \dfrac{4q^3G(q^5)^4H(q^5)^3E(q^{50})^2}{G(q^{10})^2H(q^{10})E(q^{25})^2}\cdot \frac{q^6 G(q^{10})^2 H(q^{10})^4}{G(q^{5})^2 H(q^{5})}\\
&\quad -\dfrac{2q^4G(q^5)^3H(q^5)^4E(q^{50})^2}{G(q^{10})^2H(q^{10})E(q^{25})^2}\cdot \frac{G(q^{10})^3 H(q^{10})^2 E(q^{25})^2}{G(q^{5}) H(q^{5}) E(q^{50})^2}\\
&=0.
\end{align*}
It follows that $\gamma(5n+4)$ is always zero.

\medskip

The proof of \eqref{vanisH(q^{10})-2} is analogous. We only need to multiplying \eqref{CT-iden-1} by \eqref{eq:AB-5.4} and then expand the product. The details will be omitted.

\section{Final remarks}

By similar techniques of proving \eqref{A2} and \eqref{A3}, we are able to show the following results:
\begin{align}
\sum_{n=0}^\infty\alpha(5n+1)q^n &=-\dfrac{2H(q)G(q^2)^3H(q^2)^3E(q^5)^2E(q^{20})^2}
{G(q^4)H(q^4)^2E(q^{10})^4},\\
\sum_{n=0}^\infty\alpha(5n+4)q^n &=-\dfrac{2 G(q^2)^4 H(q^2)^5 E(q^{10})^2}
{G(q) G(q^4)^2 H(q^4)^2 E(q^5)^2},\\
\sum_{n=0}^\infty\beta(5n+1)q^n &=\dfrac{2 G(q^2)^5 H(q^2)^4 E(q^{10})^2}
{H(q) G(q^4)^2 H(q^4)^2 E(q^5)^2},\\
\sum_{n=0}^\infty\beta(5n+2)q^n &=\dfrac{2G(q^2)^4H(q^2)^3E(q^{20})^2}
{G(q^4)H(q^4)^2E(q^{10})^2},\\
\sum_{n=0}^\infty\beta(5n+3)q^n &=-\dfrac{2qG(q^2)^3H(q^2)^4E(q^{20})^2}{G(q^4)^2H(q^4)E(q^{10})^2},\\
\sum_{n=0}^\infty\beta(5n+4)q^n &=-\dfrac{2 G(q) G(q^2)^3 H(q^2)^3 E(q^5)^2 E(q^{20})^2}
{G(q^4)^2 H(q^4) E(q^{10})^4}.
\end{align}
However, it seems that the generating functions for the sequences $\{\alpha(5n)\}$ and $\{\beta(5n)\}$ cannot be simplified to one single theta-quotient.

\medskip

Our numerical experiments also reveal the following sign patterns:
\begin{align}
&\alpha(n)\begin{cases}
>0, &\text{if } n\equiv0,2,5,9\pmod{10},\\
<0, &\text{if } n\equiv1,4,6,8\pmod{10},
\end{cases}\label{pattern-1}\\
&\beta(n)\begin{cases}
>0, &\text{if } n\equiv0,1,2\pmod{10},\\
<0, &\text{if } n\equiv4,5,6,8,9\pmod{10},
\end{cases}\label{pattern-2}
\end{align}
except for $\alpha(6)=\beta(6)=0$. However, the sign patterns \eqref{pattern-1} and \eqref{pattern-2} can hardly be confirmed completely by the above dissections. On the other hand, making use of an asymptotic formula provided by the first author in \cite{Che2019}, we are able to show the validity of \eqref{pattern-1} and \eqref{pattern-2} for
sufficiently large $n$.

\section*{Acknowledgement}

The second author was supported by the Postdoctoral Science Foundation of China (No.~2019M661005).

\bibliographystyle{amsplain}

\begin{thebibliography}{99}
	
	\bibitem{AG1994}
	K. Alladi and B. Gordon, Vanishing coefficients in the expansion of products of Rogers--Ramanujan type, Proc. Rademacher Centenary Conference, (G. E. Andrews and D. Bressoud, Eds.), \textit{Contemp. Math.} \textbf{166} (1994), 129--139.
	
	\bibitem{AB2005}
	G. E. Andrews and B. C. Berndt, \textit{Ramanujan's Lost Notebook. Part I}, Springer, New
	York, 2005.

	\bibitem{AB1979}
	G. E. Andrews and D. M. Bressoud, Vanishing coefficients in infinite product expansions, \textit{J. Austral. Math. Soc. Ser. A} \textbf{27} (1979), no. 2, 199--202.

	\bibitem{BB2018}
	N. D. Baruah and N. M. Begum, Exact generating functions for the number of partitions into distinct parts, \textit{Int. J. Number Theory} \textbf{14} (2018), no. 7, 1995--2011.

	\bibitem{BK2019}
	N. D. Baruah and M. Kaur, Some results on vanishing coefficients in infinite product expansions, \textit{Ramanujan J.} (2019), in press. doi: 10.1007/s11139-019-00172-x.

	\bibitem{Che2019}
	S. Chern, Asymptotics for the Taylor coefficients of certain infinite products, arXiv preprint (2019). (\href{http://arxiv.org/abs/1902.10839v1}{arXiv:1902.10839v1}).

	\bibitem{CT2019}
	S. Chern and D. Tang, Representations involving the Rogers--Ramanujan continued fraction and their applications, submitted.

	\bibitem{CT2020}
	S. Chern and D. Tang, The Rogers--Ramanujan continued fraction, Ramanujan's parameter and related 5-dissections, submitted.

    \bibitem{DX2019}
    D. Q. J. Dou and J. Xiao, The 5-dissections of two infinite product expansions, \textit{Ramanujan J.} (2019), in press. doi: 10.1007/s11139-019-00200-w.

	\bibitem{Hir1998}
	M. D. Hirschhorn, On the expansion of Ramanujan's continued fraction, \textit{Ramanujan J.} \textbf{2} (1998), no. 4, 521--527.

	\bibitem{Hir2017}
	M. D. Hirschhorn, \textit{The Power of $q$}, Developments in Mathematics Vol. 49, Springer, Cham, 2017.

	\bibitem{Hir2018}
	M. D. Hirschhorn, Two remarkable $q$-series expansions, \textit{Ramanujan J.} \textbf{49} (2018), no. 2, 451--463.

	\bibitem{Mac2015}
	J. Mc Laughlin, Further results on vanishing coefficients in infinite product expansions, \textit{J. Austral. Math. Soc. Ser. A} \textbf{98} (2015), no. 1, 69--77.

	\bibitem{McL2019}
	J. Mc Laughlin, Some observations on Lambert series, vanishing coefficients and dissections of infinite products and series, arXiv preprint (2019). (\href{http://arxiv.org/abs/1906.11978v1}{arXiv:1906.11978v1}).

	\bibitem{RS1978}
	B. Richmond and G. Szekeres, The Taylor coefficients of certain infinite products, \textit{Acta Sci. Math. (Szeged)} \textbf{40} (1978), no. 3-4, 347--369.

	\bibitem{Tan2019}
	D. Tang, Vanishing coefficients in some $q$-series expansions, \textit{Int. J. Number Theory} \textbf{15} (2019), no. 4, 763--773.

    \bibitem{Tan2020}
	D. Tang, Vanishing coefficients in four quotients of infinite product expansions, \textit{Bull. Aust. Math. Soc.} \textbf{100} (2019), no. 2, 216--224.

\end{thebibliography}

\end{document}